\documentclass[english]{amsart}
\usepackage[T1]{fontenc}
\usepackage[latin9]{inputenc}
\usepackage{color}
\usepackage{amsthm}
\usepackage{amsbsy}
\usepackage{amstext}
\usepackage{graphicx}
\usepackage{amssymb}
\usepackage{esint}

\makeatletter

\providecommand{\tabularnewline}{\\}

\numberwithin{equation}{section}
\numberwithin{figure}{section}
\theoremstyle{plain}
\newtheorem{thm}{Theorem}
  \theoremstyle{plain}
  \newtheorem{prop}[thm]{Proposition}
  \theoremstyle{plain}

\usepackage{eurosym}\usepackage{amsfonts}\usepackage{caption}

\theoremstyle{plain}

\numberwithin{equation}{section}

\makeatother

\usepackage{babel}

\begin{document}

\title{The full-tails gamma distribution applied to model extreme values}

\author{Joan del Castillo, Jalila Daoudi and Isabel Serra}
\begin{abstract}
In this article we show the relationship between the Pareto distribution
and the gamma distribution. This shows that the second one, appropriately
extended, explains some anomalies that arise in the practical use
of extreme value theory. The results are useful to certain phenomena
that are fitted by the Pareto distribution but, at the same time,
they present a deviation from this law for very large values. Two
examples of data analysis with the new model are provided. The first
one is on the influence of climate variability on the occurrence of
tropical cyclones. The second one on the analysis of aggregate loss
distributions associated to operational risk management.
\end{abstract}
\maketitle

\subsubsection*{\textbf{\emph{Keywords}}: Exponential models. Heavy tailed distributions.
Pareto distribution. Power-law distribution. Type III distribution.
Operational risk models.}

\section{Introduction}

The extreme value theory is used by many authors to model exceedances
in several fields such as hydrology, insurance, finance and environmental
science, see Furlan (2010), Coles and Sparks (2006), Moscadelli (2004).
However, the theory shows some surprises in practical applications.
For instance, Dutta and Perry (2006) observed, in an empirical analysis
of models for estimating operational risk, that even when Pareto distribution
fit the data it may result in unrealistic capital estimates (sometimes
more than 100\% of the asset size), see also Degen, \textit{et al}.
(2007). In other instances despite being well-founded the power law-distribution,
as in Corral, \textit{et al}. (2010), it may happen that it works
in the central region but not for larger values. These challenges
should motivate us to find new models that describe the characteristics
of the data rather than limit the data so that it matches the characteristics
of the model (Dutta and Perry, 2006).

The peaks over threshold (PoT) method for estimating high quantiles
is based on the Pickands-Balkema-DeHaan Theorem, see McNeil, \textit{et
al}. (2005) and Embrechts, \textit{et al}. (1997). Hence, in practice,
the conditional distribution of any random variable over a high threshold
is approximated by a generalized Pareto distribution (GPD). This result
is a mathematical solution to the question, but the practical problem
whether the threshold is high enough still remains. 

In this paper a new statistical approach for estimating high quantiles
is provided for non-light tails data sets. It is shown that the Pareto
distribution is nested in the statistical model here called \emph{full-tails
gamma} (FTG) distribution. FTG model is a scale parameter family of
distributions on $\left(0,\infty\right)$ closed by truncation, Hence,
it allows us to find distributions so close to the Pareto distribution
as determined by the data, but with greater flexibility, extending
the distributions for non-light tails provided by GPD. With the current
specialized computer programs for statistical analysis is not difficult
to deal with the FTG distribution, since the incomplete gamma function
and its derivatives are now easily available, see Abramowitz and Stegun
(1972). The work of pioneers like Chapman (1956) must be viewed in
this way. Another approach for lighter tails is in Akinsete, \textit{et
al}. (2008). 

The FTG distribution is related to very old families of distributions
as the Pareto III distribution, see Arnold (1983, pp 3) and Davis
\emph{et al} (1979). The gamma distribution is one of the most studied
families of distributions, since Fisher (1922). For life theory and
reliability the two-parameter right truncated gamma distribution is
usually considered since Chapman (1956). Den Broeder (1955) considered
the left truncated gamma distribution but with known scale parameter.
Stacy (1962) introduced a three-parameter generalized gamma distribution
which includes, as special cases, the two-parameter gamma and the
two-parameter Weibull. Harter (1967) extends the model to a four-parameter
family, by including a location parameter. Hedge and Dahiya (1989)
obtain necessary and sufficient conditions for the existence of the
MLE of the parameters of a right truncated gamma distribution. The
right truncated gamma with unknown origin is studied by Dixit and
Phal (2005). For simulation of right and left truncated gamma distributions,
see Philippe (1997). Also in physics literature the FTG distribution
appears related to the power-law with (exponential) cut off, see Clauset
et al (2009) or Sornette (2006), however, the models are not the same.

In Section 2, FGT distribution is introduced, showing that the domain
of parameters includes the gamma distribution and the Pareto distribution
(Theorem 1) in the boundary. Proposition 2 provides a clear interpretation
of its three parameters $\left(\alpha,\theta,\rho\right)$. The FTG
distribution for $\alpha>0$ is the left truncated gamma distribution
relocated to the origin . The FTG distribution for $\alpha\leq0$
appears as the full exponential model generated from a canonical statistic.
Section 3 describes the most basic statistical properties of the FTG,
as the moments generating function, a simulation method and the standard
tools for MLE. 

In Section 4, we provide applications of the FTG exemplify that are
usually fitted by Pareto distribution. The first one on the influence
of climate variability and global warming on the occurrence of tropical
cyclones, see Corral, \textit{et al}. (2010). Here, classical goodness
of fit test rejects Pareto distribution but it offers no alternative
to that model. The alternative is here provided by the FTG distribution.

The second example deals with the analysis of aggregate loss distributions
associated to operational risk management, see Degen, \textit{et al}.
(2007). The concept of operational risk is founded in the Basel II
accord of 1999, that has been widely adopted around the world as a
regulatory requirement by central banks. The focus on systemic risk,
precipitated by the current crisis, has elevated operational risk
management to greater prominence. Risk capital, under the PoT approach,
has been calculated here with Pareto and FTG distributions. Table
3 shows that Pareto distribution provides unrealistic and highly unstable
estimations, however, FTG distribution provides more realistic and
much more stable risk capital estimations.

\section{The full-tails gamma distribution}

The\emph{ FTG }distribution is the three-parameter family of continuous
probability distributions, with support on $(0,\infty)$, defined
by $\alpha\in\mathbb{R},\theta>0,\rho>0$ by

\begin{equation}
f\left(x;\alpha,\theta,\rho\right)=\theta\left(\rho+\theta x\right)^{\alpha-1}\exp(-(\rho+\theta x))/\Gamma(\alpha,\rho).\label{FTG}\end{equation}
where $\Gamma\left(\alpha,\rho\right)$ is the (upper) incomplete
gamma function, see Abramowitz and Stegun (1972), \begin{equation}
\Gamma\left(\alpha,\rho\right)=\int_{\rho}^{\infty}t^{\alpha-1}e^{-t}dt,\label{ing}\end{equation}
in particular $\Gamma\left(\alpha,0\right)=\Gamma\left(\alpha\right)$
is the gamma function. The FTG distribution extends to some boundary
parameters as shall seen bellow.

If $\alpha>0,\theta>0$ and $\rho=0$, the family\ $\left(\ref{FTG}\right)$
clearly extends to the probability density function of the gamma distribution,
defined by\begin{equation}
g\left(x;\alpha,\theta\right)=\theta^{\alpha}x^{\alpha-1}\exp\left(-\theta x\right)/\Gamma\left(\alpha\right)\label{gam}\end{equation}

For $\alpha>0$ the FTG\  is the left truncated gamma distribution
relocated to the origin, or equivalently it is a tail of gamma distribution.
Note that in this paper,  \emph{tail} is used in sense of conditional
exceedances over a threshold.  More exactly, suppose $F(x)$ is a
absolutely continuous\ cumulative distribution function of a non
negative random variable, $X$. The probability density function of
the exceedances of $X$ at $u>0$ is defined by\begin{equation}
f_{u}\left(x\right)=f\left(x+u\right)/(1-F(u))\label{thr}\end{equation}
where $f\left(x\right)=F^{\prime}\left(x\right)$. If $\alpha>0$
and $\rho>0$, then $\left(\ref{FTG}\right)$ is the probability density
function of the exceedances of a gamma distribution at $\sigma>0$,
with $\sigma=\rho/\theta$ \begin{equation}
g_{\sigma}\left(x;\alpha,\theta\right)=\theta^{\alpha}\left(x+\sigma\right)^{\alpha-1}\exp\left(-\theta\left(x+\sigma\right)\right)/\Gamma\left(\alpha,\sigma\theta\right).\label{trg}\end{equation}

The probability density function of the Pareto distribution, defined
for $x>0$, is\begin{equation}
p\left(x;\alpha,\sigma\right)=-\alpha\sigma^{-1}\left(1+x/\sigma\right)^{\alpha-1},\label{par}\end{equation}
where $\alpha<0$ and $\sigma>0$.  This parameterization will be
used to show that the Pareto distribution appears in this boundary
of the FTG distribution.

\begin{thm}
\label{pro:3}Let $\sigma=\rho/\theta>0$ fixed in $\left(\ref{FTG}\right)$
and $\alpha<0$. If $\rho$ tends to zero, then the probability density
function $\left(\ref{FTG}\right)$ tends to the probability density
function of the Pareto distribution\emph{ }(\ref{par}) in $L^{1}$
norm. Moreover, the convergence extends to the moments, provided the
corresponding moments for Pareto distribution are finite. \end{thm}
\begin{proof}
Observe that if $\rho$ tends to $0$, then $\theta$ tends to $0$,
since $\sigma=\rho/\theta>0$ is fixed. Using $\theta=\rho/\sigma$,
\[
f(x;\alpha,\theta,\rho)=\rho^{\alpha}\sigma^{-1}(1+x/\sigma)^{\alpha-1}\exp(-\rho(1+x/\sigma))/\Gamma(\alpha,\rho)\]
converges pointwise to the probability density function $\left(\ref{par}\right)$,
since the property (5.1.23) of Abramowitz and Stegun (1972), under
the assumptions,\begin{equation}
\lim_{\rho\to0}\rho^{\alpha}/\Gamma(\alpha,\rho)=-\alpha\label{lim}\end{equation}
holds. Observe that for $\rho$ small\begin{eqnarray*}
f(x;\alpha,\theta,\rho) & = & \rho\sigma^{-1}(1+x/\sigma)^{\alpha-1}\exp(-\rho(1+x/\sigma))/\Gamma(\alpha,\rho)\\
 & \leq & -2\alpha\sigma^{-1}(1+x/\sigma)^{\alpha-1}=2p\left(x;\alpha,\sigma\right)\end{eqnarray*}
since from the limit $\left(\ref{lim}\right)$ we can consider the
boundedness $\rho^{\alpha}/\Gamma(\alpha,\rho)\leq-2\alpha$. Finally,
from the dominated convergence theorem we obtain the convergence in
$L^{1}$. Moreover, whenever the moments of Pareto distribution are
finite, the convergence extends to these moments.
\end{proof}
Therefore, the family $\left(\ref{FTG}\right)$ has the boundary
parameter sets corresponding to the gamma distribution:$\left\{ \alpha>0,\theta>0,\rho=0\right\} $
and the Pareto distribution:$\left\{ \alpha<0,\theta=0,\sigma>0\right\} $.
\textcolor{blue}{} Summarizing, the FTG distribution\textcolor{red}{$\left(\ref{FTG}\right)$}
includes the gamma distribution, the truncated gamma distribution
$\left(\alpha>0\right)$, its extension to  $\alpha{\color{blue}\leq}0$
and the Pareto distribution, see Figure~\ref{graf_densities}. 

\begin{figure}[h]
\centering{}\centerline{\includegraphics[bb=0bp 302bp 748bp 595bp,clip,width=13cm]{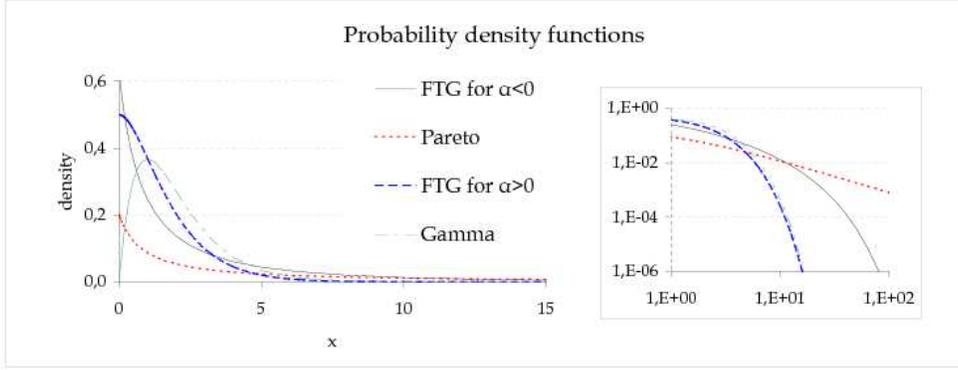}}
\caption{The left shows some probability density functions in FTG family. The
FTG with $\alpha=2$, $\sigma=1$ and $\theta=1$ corresponds to the
tails of a gamma and the FTG for $\alpha=-0.2$, $\sigma=1$ and $\theta=0.1$
is not. For the boundary parameter sets of the family, we consider
gamma with $\alpha=2$ and $\theta=1$ and Pareto with $\alpha=-0.2$
and $\sigma=1$. The right shows the same plot in common logarithm
for the tail of the functions. We see exponential decay except for
Pareto probability density function. \label{graf_densities}}
\end{figure}

\begin{prop}
\label{proptrsc}Let $X$ be a random variable distributed as $FTG(\alpha,\theta,\rho)$,
then
\end{prop}
(a) For $\lambda>0$, the random variable $\lambda X$ is distributed
as $FTG\left(\alpha,\theta/\lambda,\rho\right)$.

(b) For any threshold, $u>0$, the threshold exceedances, $X_{u}$
is distributed as $FTG(\alpha,\theta,\rho+\theta u).$
\begin{proof}
The first result holds from the probability density function of $\lambda X$
for $\lambda>0$, \[
f\left(x/\lambda;\alpha,\theta,\rho\right)/\lambda=(\theta/\lambda)\left(\rho+\theta x/\lambda\right)^{\alpha-1}\exp(-(\rho+\theta x/\lambda))/\Gamma(\alpha,\rho)=f(x;\alpha,\theta/\lambda,\rho)\]
remark that for $\theta=0$ is\[
p\left(x/\lambda;\alpha,\sigma\right)/\lambda=-\alpha(\lambda\sigma)^{-1}\left(1+x/(\lambda\sigma)\right)^{\alpha-1}=p(x;\alpha,\lambda\sigma).\]
And the second one is a consequence of $\left(\ref{thr}\right).$
For $\theta>0$ is\[
\frac{f\left(x+u;\alpha,\theta,\rho\right)}{1-F(u)}=\frac{\theta\left(\rho+\theta u+\theta x\right)^{\alpha-1}}{\Gamma(\alpha,\rho+\theta u)}\exp(-(\rho+\theta u+\theta x))=f(x;\alpha,\theta,\rho+\theta u)\]
and for $\theta=0$
\end{proof}
From the last result it is clear that the FTG distribution $\left(\ref{FTG}\right)$
is a scale parameter family on $\left(0,\infty\right)$ closed by
truncation, in the sense of $\left(\ref{thr}\right)$. Hence it is
appropriate for modelling (non-light) tails of datasets in the sense
Balkema-DeHaan (1974) and Pickands (1975), since it contains the Pareto
distribution and the exponential distribution. The parameter $\beta=1/\theta$
is the scale parameter and the parameter $\rho$ is the truncation
parameter. The parameter $\left(-\alpha\right)$ shall be interpreted
in terms of the Pareto distribution as the weight of the tail. Then,
each one of the three-parameter separately has a clear interpretation. 

Given $\sigma$ fixed, the family $\left(\ref{FTG}\right)$ for $\alpha\leq0$
appears as the full exponential model generated from a canonical statistic
$\left(x,log\left(\sigma+x\right)\right)$, see Barndorff-Nielsen
(1978), Brown (1986), Letac (1992). 

The FTG distribution is related to the three-parameter model Pareto
type III from Arnold (1983), caracterized by the survivor function 

\begin{equation}
\bar{F}(t)=(1+t/\phi)^{-\lambda}\exp(-\theta t).\label{surp3}\end{equation}
Moreover, the Pareto type III has been considered as a model for survival
data, Davis (1979). This fact is natural since the Pareto type III
model is a mixture of two FTG distributions. In fact, Pareto type
III model is a particular case of the six-parameter mixture model
of two FTG models.

Finally, $\left(\ref{FTG}\right)$ can also be seen as a weighted
version of the Pareto distribution, with the weight $w(x)=\exp(-\theta x)$,
that is also known as an \emph{exponencial tilting} of the distribution,
see Barndorff-Nielsen and Cox (1994).\textcolor{red}{}

\section{statistical tools and MLE}

With the current specialized computer programs for statistical analysis
is not difficult to deal with the FTG distribution. The incomplete
gamma function, $\Gamma\left(\alpha,\rho\right)$ and its derivatives
are now easily available. Symbolic differentiation allows us to get
the moments of a distribution from the moment generating function.
Simulation and optimization algorithms are available in the same way.
The work of pioneers like Chapman (1956) must be viewed in this way.

The cumulative distribution function corresponding to the family $\left(\ref{FTG}\right)$
is\[
F\left(x;\alpha,\theta,\rho\right)=1-\Gamma(\alpha,\rho+\theta x)/\Gamma(\alpha,\rho)\]
and for the Pareto distribution we have to consider the limit case,
corresponding to $P\left(x;\alpha,\sigma\right)=1-(1+x/\sigma)^{\alpha}$.

The FTG distribution has moment-generating function in the interior
of the domain of parameters. Hence, it is possible to calculate the
moments of all orders. In addition it is also possible to calculate
the moments of the conditional distribution over a threshold, by Proposition
\ref{proptrsc}. For $\alpha\in\mathbb{R},\theta>0,\rho>0$, the moment-generating
function of the FTG distribution $\left(\ref{FTG}\right)$ exist and
it is given by \begin{equation}
M(t)=M(t;\alpha,\theta,\rho)=\left(1-t/\theta\right)^{-\alpha}\exp\left(-\rho t/\theta\right)\Gamma(\alpha,\rho(1-t/\theta))/\Gamma(\alpha,\rho),\quad t<\theta.\label{mgf-1}\end{equation}
For $\alpha>0,$ it extends for $\rho=0$ and\ coincides with the
moment generating function of gamma distribution $M_{g}\left(t\right)=\left(1-t/\theta\right)^{-\alpha}$

The cumulant generating function is given by\[
K\left(t\right)=\log\left(M(t)\right)=-t\rho/\theta-\alpha\log\left(1-t/\theta\right)-\log\Gamma(\alpha,\rho)+\log\Gamma\left(\alpha,\left(1-t/\theta\right)\rho\right)\]
hence, the first moments are\begin{eqnarray*}
E\left[X\right] & = & K^{\prime}\left(0\right)=(\alpha-\rho+\mu)/\theta\\
Var\left[X\right] & = & K^{\prime\prime}\left(0\right)=(\alpha+(1+\rho-\alpha)\mu-\mu^{2})/\theta^{2}\end{eqnarray*}
where $\mu=e^{-\rho}\rho^{\alpha}/\Gamma(\alpha,\rho)$. Notice that
using the Proposition \ref{proptrsc}, to calculate the conditional
expectation for any threshold fixed $u>0$, is the same as to calculate
the expectation with modified parameters \begin{eqnarray}
E\left[X\:|\: X>u\right] & =(\alpha-\rho+\mu')/\theta\label{eq:cond}\end{eqnarray}
where $\mu'=e^{-(\rho+\theta u)}(\rho+\theta u)^{\alpha}/\Gamma(\alpha,\rho+\theta u)$
.

\subsection{Random variates generation}

Simulation methods for Pareto and gamma distributions are well known.
Has also been well studied the simulation of truncated gamma distribution
(\ref{trg}), see Philippe (1997). Hence, only the set of parameters
$\{\alpha<0,\theta>0,\rho>0\}$ for FTG distribution is considered
here.

A simple way to simulate the distribution is the inversion method,
since the cumulative distribution function has an easy expression,
however, it needs to use complex numerical processes using the incomplete
gamma function.

A simple and efficient method from numerical point of view is obtained
with an idea from Devroye (1986) on a generalization of the rejection
method. We emphasize the simplicity of this algorithm, since it does
not require the use of the incomplete gamma function.

First of all, since $1/\theta$ is a scale parameter it is enough
consider simulations for $\theta=\rho$. That is, to simulate $FTG\left(\alpha,\theta,\rho\right)$,
we can first simulate $FTG\left(\alpha,\rho,\rho\right)$ and finally
we apply the change of scale to the random sample. 

For $\theta=\rho$, the probability density function $\left(\ref{FTG}\right)$
split in three terms \begin{equation}
f\left(x;\alpha,\rho,\rho\right)=\left(\rho^{\alpha-1}e^{-\rho}/\Gamma(\alpha,\rho)\right)(\rho e^{-\rho x})\left(1+x\right)^{\alpha-1}=cg(x)\psi(x)\label{ftgSep}\end{equation}
where the function $\psi\left(x\right)=\left(1+x\right)^{\alpha-1}$
is $[0,1]$-valued, $g\left(x\right)=\rho e^{-\rho x}$ is a probability
density function easy to simulate and $c$ is a normalization constant
at least equal to 1. 

The rejection algorithm for this case can be rewritten as follows.
Generate independent random variates $\left(X,U\right)$ where $X$
has probability density function $g\left(x\right)$ and $U$ is uniformly
distributed in $[0,1]$ until $U\leq\psi(X)$. This method produces
a random variable $X$ with probability density function $f(x)$,
(Devroye, 1986).

The following code applies the method to our case, see R Development
Core Team (2010).

\begin{verbatim}
     #to generate a sample of size n of FTG(a,t,r)
     rFTG<-function(n,a,t,r) { 
          sample<-c(); m<-0 
          while (m<n) { 
               x<-rexp(1,rate=r);u<-runif(1) 
               if (u<=(1+x)^(a-1)) sample[m+1]<-x 
               m<-length(sample) } 
          sample*r/t } 
\end{verbatim}

\subsection{Maximum likelihood estimates of the parameters}

In (\ref{FTG}), FTG distribution has been introduced with parameters
$(\alpha,\theta,\rho)$, since each one separately has a clear interpretation.
For MLE estimation it is better to use $(\alpha,\sigma,\rho)$, with
dispersion parameter $\sigma=\rho/\theta$, since fixed $\sigma$
the FTG distribution is an exponential model. Hence, from Barndorff-Nielsen
(1978), it is known that the maximum likelihood estimator exist and
it is unique. The summary of the procedure to compute the MLE is
search the dispersion parameter and then to optimize the problem for
others.

Let $\boldsymbol{\textrm{x}}=\{x_{1},...,x_{n}\}$ be a of size $n$,
the log-likelihood function for FTG distribution is

\begin{equation}
l(\alpha,\sigma,\rho)=-n\left(\log\Gamma\left(\alpha,\rho\right)+\log\left(\sigma\rho^{-\alpha}\right)-\frac{\alpha-1}{n}\sum_{i=1}^{n}\log\left(1+\frac{x_{i}}{\sigma}\right)+\frac{\rho}{n}\sum_{i=1}^{n}\left(1+\frac{x_{i}}{\sigma}\right)\right)\label{like-1}\end{equation}
To simplify, we denote\begin{equation}
d=d(\alpha,\rho)=\log\Gamma(\alpha,\rho)\label{defd-1}\end{equation}
and we consider $(r,s)$, for $\sigma$ fixed as

\[
r(x;\sigma)=\left(1+x/\sigma\right)\:\:\textrm{\textrm{and}\:\:}s(x;\sigma)=\log\left(1+x/\sigma\right)\]
which are the \textit{sufficient statistics} from exponential model
point of view and then we denote the sample means as\[
\bar{r}(\boldsymbol{\textrm{x}};\sigma)=\frac{1}{n}\sum_{i=1}^{n}\left(1+x_{i}/\sigma\right)\:\:\textrm{\textrm{and}}\:\:\bar{s}(\boldsymbol{\textrm{x}};\sigma)=\frac{1}{n}\sum_{i=1}^{n}\log\left(1+x_{i}/\sigma\right).\]
To simplify, we use the parameters in subscript to denote the partials
derivatives and we omit the dependence of the parameters in these
derivatives. Hence, the scoring is $\left(l_{\alpha},l_{\sigma},l_{\rho}\right)$
and it is given by \begin{eqnarray}
l_{\alpha} & = & -n\left\{ d_{\alpha}-\log\left(\rho\right)-\bar{s}(\boldsymbol{\textrm{x}};\sigma)\right\} \label{logeq:1}\\
l_{\sigma} & = & -n\left\{ \sigma^{-1}-(\alpha-1)\bar{s}_{\sigma}+\rho\bar{r}_{\sigma}\right\} \label{logeq:2}\\
l_{\rho} & = & -n\left\{ d_{\rho}-\alpha\rho^{-1}+\bar{r}(\boldsymbol{\textrm{x}};\sigma)\right\} \label{logeq:3}\end{eqnarray}
the observed information matrix is given by

\[
I_{O}(\alpha,\sigma,\rho)=-n\left(\begin{array}{ccc}
d_{\alpha\alpha} & -\bar{s}{}_{\sigma} & d_{\alpha\rho}-\rho^{-1}\\
\\-\bar{s}_{\sigma} & -\sigma^{-2}-(\alpha-1)\bar{s}{}_{\sigma\sigma}+\rho\bar{r}_{\sigma\sigma} & \bar{r}_{\sigma}\\
\\d_{\alpha\text{\ensuremath{\rho}}}-\rho^{-1} & \bar{r}_{\sigma} & d_{\rho\rho}+\alpha\rho^{-2}\end{array}\right)\]
and it can be used to compute the confidence interval for the maximum
likelihood estimates $\hat{\alpha},\:\hat{\sigma}\:\textrm{and\:}\hat{\rho}$
of the parameters $\alpha,\:\sigma\:\textrm{and\:}\rho,$ respectively.

To compute the MLE is convenient to solve the equation (\ref{logeq:2})
for to get $\hat{\sigma}$ using $(\hat{\alpha}(\sigma),\hat{\rho}(\sigma))$
for the parameters $(\alpha,\rho)$ or, more general, to maximize
the profile log-likelihood equation

\begin{equation}
l_{p}(\sigma)=-n\left(\:\log\Gamma\left(\hat{\alpha}(\sigma),\hat{\rho}(\sigma)\right)+\log\left(\sigma\hat{\rho}(\sigma)^{-\hat{\alpha}(\sigma)}\right)-(\hat{\alpha}(\sigma)-1)\bar{s}(\boldsymbol{\textrm{x}},\sigma)+\hat{\rho}(\sigma)\bar{r}(\textrm{\ensuremath{\boldsymbol{\textrm{x}}}},\sigma)\right)\label{like-prof}\end{equation}
where $(\hat{\alpha}(\sigma),\hat{\rho}(\sigma))$ is the only one
solution of the system in $(\alpha,\rho)$ consists of the equations
(\ref{logeq:1}) and (\ref{logeq:3}) for $\sigma$ fixed. Remark
that, from a practical point of view, is convenient to consider this
pair of equations to simplify the equation (\ref{like-prof}) (or
(\ref{logeq:2})) in an equation as light as possible of the sample
explicitly. For instance, the equation (\ref{like-prof}) can be
simplified by

\begin{equation}
l_{p}(\sigma)=-n\left(\:\log\Gamma\left(\hat{\alpha}(\sigma),\hat{\rho}(\sigma)\right)-\log\left(\hat{\rho}(\sigma)\sigma^{-1}\right)-(\hat{\alpha}(\sigma)-1)d_{\alpha}-\hat{\rho}(\sigma)d_{\rho}+\hat{\alpha}(\sigma))\right)\label{like-prof-1}\end{equation}
remark that is an expression without the sample explicitly.

A procedure to obtain the MLE in R is computing the MLE of the standardized
sample $\textrm{\ensuremath{\boldsymbol{\textrm{y}}}}=\{x_{i}/\bar{x}\}_{1\leq i\leq n}$,
considering the initial estimates as follows. We have two options:
to take the initial estimates as $(\dot{\alpha},1,\dot{\theta})$
where $(\dot{\alpha},\dot{\theta})$ is the MLE of gamma model or
take the initial estimates as $(\dot{\alpha},\dot{\sigma},\dot{\rho})$
where $(\dot{\alpha},\dot{\sigma})$ is the MLE of Pareto model and
$\dot{\rho}$ is obtained by the relation (from the equation (\ref{logeq:3}))\[
d_{\rho}-\dot{\alpha}\dot{\rho}^{-1}+1+\dot{\sigma}^{-1}=0.\]
Finally, $(\hat{\alpha},\hat{\sigma},\hat{\rho})$ (the MLE for the
sample $\boldsymbol{\textrm{x}}$) is obtained using the Proposition
\ref{proptrsc}, in fact we obtain $\hat{\alpha}=\hat{\alpha'},\:\hat{\sigma}=\hat{\sigma'}/\bar{x}\:\textrm{and\:}\hat{\rho}=\hat{\rho'}$.

 Finally, it might be appropriate to consider de log-scale for $\sigma$
and $\rho$. R has a package to optimize which greatly simplify the
calculation the MLE.

\section{Data Analysis}

Certain phenomena that may be fitted by Pareto distribution, or the
power-law distribution, present a deviation from these laws for very
large values. It is often due to the interference that produces an
overall limit (a finite ocean basin or a loss limited to the total
value of a economy). The motivation of this work was to find a model
to explain this fact in several cases, such as the energy of tropical
cyclones or the calculation of regulatory capital for operational
risk.\textcolor{red}{{} .}

In the first example, Choulakian and Stephens (2001) goodness of fit
test rejects the Pareto distribution, but no alternative is provided.
We show that FTG is a better model fitting even the very large values.
In the second example, goodness of fit test can not be applied, since
the parameter is outside the range of parameters provided by their
tables. However, FTG is a better model that Pareto distribution, providing
more realistic and much more stable risk capital estimations.

\subsection{\textcolor{blue}{Analysis of tropical cyclones}}

Corral, \textit{et al}. (2010) study the influence of climate variability
and global warming through the occurrence of tropical cyclones. Their
approach is based on the application of an estimation of released
energy to individual tropical cyclones. We are going to compare our
model with its statistical analysis on power-law distribution for
$494$ tropical cyclones occurred in the North Atlantic between 1966
and 2009. 

To measure the importance of the tropical cyclones it is used an estimation
of released energy, the power dissipation index (PDI), defined by
\[
PDI=\sum_{t}v_{t}^{3}\Delta t\]
where $t$ denotes time and runs over the entire lifetime of the storm
and $v_{t}$ is the maximum sustained surface wind velocity at time
$t$ (PDI units are $m^{3}/s^{2}$). The PDI of the original data
is between $5.38\:10^{8}$ and $2.54\:10^{11}$. Deviations from the
power law at small PDI values were attributed to the deliberate incompleteness
of the records for `no significant' storms. Their estimation only
considers tropical cyclones with PDI bigger than $3\:10^{9}$, that
is a sample of size $372$ ($75\%$ of the original data).

Figure~\ref{ex1_hist} shows the fit of the power-law distribution
with an empirical aproximation probability density function of the
sample. Given the sample of tropical cyclones $\{x_{i}\}$ for $1\leq i\leq n$
with $n=494$, Corral, \textit{et al}. (2010) approximate the probability
density function at points $p_{r}=10^{8+(r-1)/5}$ for $1<r<m$ where
$m=21$, for the histogram values \[
h_{r}=\frac{{\#\{\: x_{i}\::\: l_{r}<x_{i}\leq l_{r+1}\}}}{n\:(l_{r+1}-l_{r})}\]
 for the intervals given by $l_{s}=0.5\:10^{8+s/5}\:11^{1/5}$ with
$1<s<m+1$. The goal of their method is to plot in common logarithm
(base 10) scale for both axes, since the power-law probability density
function in this situation corresponds to a straight line. The fit
is done by minimum square method for a set of points $\{(u_{r},v_{r})\}$,
where $u_{r}=\log_{10}p_{r}$ and $v_{r}=\log_{10}h{}_{r}$. 

Our first contribution consists in fitting the FTG distribution by
MLE for the whole sample. The FTG distribution shows a really best
fit especially in the tail of the data, see Figure ~\ref{ex1_hist}.
The more rapid decay at large PDI is associated with the finite size
of the ocean basin. That is, the storms with the largest PDI do not
have enough room to last a longer time. The relevant thing is that
FTG distribution fits the date even in this situation.

\begin{figure}[th]
\centering{}\centerline{\includegraphics[bb=0bp 300bp 700bp 595bp,clip,width=13cm]{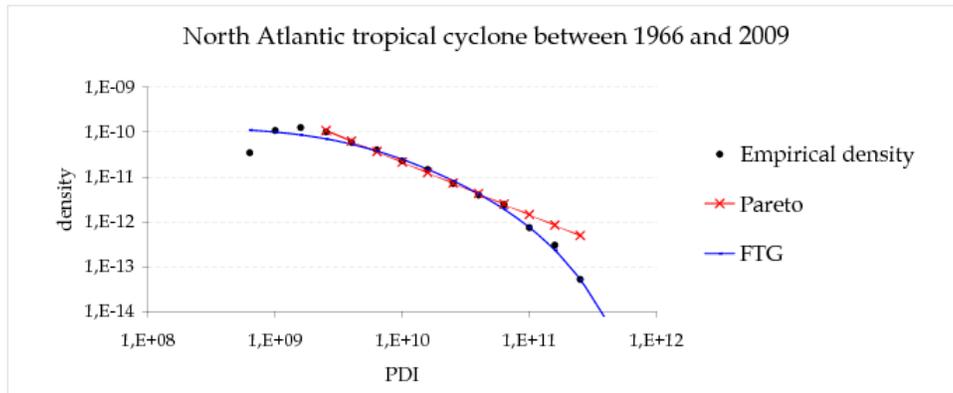}}
\caption{FTG distribution fits better than Pareto distribution the tropical
cyclones data set, especially in the tail of the observations. The
plot is scaled in common logarithm for both axes.\label{ex1_hist}}
\end{figure}

Theorem \ref{pro:3} shows that Pareto distribution is nested in FTG
distribution, hence likelihood inference is now available. MLE of
parameters and its standard deviations are shown in Table~\ref{ex1_ajusts}
for FTG distribution and Pareto distribution (with two parameters).
The values of log-likelihood function are $-667.58$ for FTG case
(truncated gamma distribution) and $-680.06$ for the Pareto case. 

First of all, the goodness of fit test for Pareto distribution given
by Choulakian and Stephens (2001) rejects with p-value less than 0.001
for both statistics, $W^{2}=0.28$ and $A^{2}=2.4$, of the method,
but it offers no alternative to the model. Finally, the likelihood
ratio test can be used to find a confidence region around the FTG
parameters, concluding that the difference between the FTG and the
Pareto distribution is highly significant. The $p$-value is $5.8\:10^{-7}$.

\begin{table}[h]
\begin{centering}
\begin{tabular}{c|ccc|cccc|l|}
\cline{2-9} 
 & \multicolumn{3}{c|}{Pareto distribution} & \multicolumn{4}{c|}{FTG distribution} & \tabularnewline
 & $\alpha$ & \multicolumn{1}{c|}{$\sigma$} & $l$ & $\alpha$  & $\sigma$  & \multicolumn{1}{c|}{$\rho$ } & $l$ & LRT\tabularnewline
\hline 
\multicolumn{1}{|c|}{MLE } & -1.63 & 2.01 & -680.06 & 0.28  & 0.09  & 0.02  & -667.58 & 24.96\tabularnewline
\multicolumn{1}{|c|}{s.e. } & 0.22 & 0.41 &  & 0.15  & 0.11  & 0.02  &  & \tabularnewline
\hline
\multicolumn{1}{c}{} &  &  & \multicolumn{1}{c}{} &  &  &  & \multicolumn{1}{c}{} & \multicolumn{1}{l}{}\tabularnewline
\end{tabular}
\par\end{centering}

\centering{}\caption{MLE for FTG and Pareto distributions for tropical cyclone occurred
in the North Atlantic between 1966 and 2009. The data used corresponds
to PDI over $3\:10^{9}$, with the origin shifted to zero (units are
$10^{10}m^{3}/s^{2}$). This change does not afect the likelihood
ratio test, LRT, and the $\alpha$ parameter.\label{ex1_ajusts}}
\end{table}

\subsection{Analysis of aggregate loss distributions}

Financial institutions use internal and external loss data in order
to compare several approaches for modelling \emph{aggregate loss distributions},
associated to quantitative modelling of operational risk, see Dutta
and Perry (2006), Degen, \textit{et al}. (2007) and Moscadelli (2004).
The data used for the analysis was collected by several banks participating
in the survey to provide individual gross operational losses above
a threshold, starting on 2002. The data was grouped by eight standardized
\emph{business lines} and seven \emph{event types}. 

\emph{Risk capital} is measured as the $99.9\%$ percentile level
of the simulated capital estimates for aggregate loss distributions
in holding period (1 year). A \emph{loss event} $L_{i}$ (also known
as the loss severity) is an incident for which an entity suffers damages
that can be measured with a monetary value. An aggregate loss over
a specified period of time can be expressed as the sum

\begin{equation}
S=\sum_{i=1}^{N}L_{i}\label{eq:ald}\end{equation}
where $N$ is a random variable that represents the frequency of losses
that occur over the period. As usual, here it is assumed that the
$L_{i}$ are independent and identically distributed, and each $L_{i}$
is independent from $N$, that is Poisson distributed, with parameter
$\lambda$.

The data set used here correspond to the $40$ largest losses associated
with the business line \emph{corporate finance} and the event type
\emph{external fraud}, observed over a high threshold, $u$. To maintain
confidentiality, the data $\left\{ x_{j}\right\} $ has been scaled
to threshold zero and mean $100$, according to\[
y_{j}=100\left(\frac{x_{j}-u}{\bar{x}-u}\right)\]
The $40$ exceedances, rounded to two decimal place, were: $0.07$,
$0.11$, $0.26$, $0.40$, $0.46$, $0.62$, $0.70$, $0.75$, $0.89$,
$1.08$, $1.52$, $1.64$, $1.69$, $2.04$, $2.19$, $2.52$, $2.73$,
$3.16$, $3.74$, $4.04$, $4.63$, $5.44$, $5.86$, $6.02$, $10.32$,
$19.63$, $29.13$, $30.36$, $30.88$, $35.78$, $40.07$, $46.12$,
$137.52$, $237.05$, $311.14$, $314.19$, $396.29$, $552.48$,
$864.88$, $891.62$.

Aggregate losses are determined mainly by the extreme values of loss
events distribution. In this case, risk capital depends on $40$ exceedances,
but, to calculate the $99.9\%$ quantile, a model is required. Under
the PoT approach extreme values are modelled with Pareto distribution,
see Degen, \textit{et al}. (2007) and Moscadelli (2004). Pickands-Balkema-DeHaan
theorem justifies the approach, see McNeil, \textit{et al}. (2005).
However, this approach may result in unrealistic capital estimates,
especially when the fitted Pareto distribution has infinite expectation.

Since the data set has only exceedances over a threshold, the PoT
method is the appropriate way. When all losses are recorded, Dutta
and Perry (2006) use a four-parameter distribution, called g-and-h,
to model the data. If we focus on extreme events of financial assets
returns, both upside and downside, standard methodologies also include
the classical Student's t and stable Paretian distributions, see Rachev,
\textit{et al}. (2010). 

\begin{table}[h]
\begin{centering}
\begin{tabular}{c|ccc|cccc|l|}
\cline{2-9} 
 & \multicolumn{3}{c|}{Pareto distribution} & \multicolumn{4}{c|}{FTG distribution} & \tabularnewline
 & $\alpha$ & \multicolumn{1}{c|}{$\sigma$} & $l$ & $\alpha$  & $\sigma$  & \multicolumn{1}{c|}{$\rho$ } & $l$ & LRT\tabularnewline
\hline 
\multicolumn{1}{|c|}{MLE } & -0,45 & 1,38 & -174,44 & -0,20 & 0.65 & 4.3E-4 & -172,37 & 4,14\tabularnewline
\multicolumn{1}{|c|}{s.e. } & 0.10 & 0.73 &  & 0.16 & 0.59 & 6.2E-4 &  & \tabularnewline
\hline
\multicolumn{1}{c}{} &  &  & \multicolumn{1}{c}{} &  &  &  & \multicolumn{1}{c}{} & \multicolumn{1}{l}{}\tabularnewline
\end{tabular}
\par\end{centering}

\centering{}\caption{MLE for FTG and Pareto distributions of losses by external fraud.
FTG distribution is a better model than Pareto distribution, from
likelihood ratio test. \label{ex2_ajusts}}
\end{table}

Table \ref{ex2_ajusts} gives the MLE of parameters for Pareto and
FTG distributions, as well as its standard deviations and log-likelihood
function, for the last data set. First of all we observe that for
Pareto distribution the parameter is in the range $0<\left(-\alpha\right)<1$,
that is, a distribution with infinite expectation. This can not be
rejected with the goodness of fit test for Pareto distribution given
by Choulakian and Stephens (2001), since the parameter is outside
the range of parameters provided by their tables. However, Pareto
distribution is nested in FTG distribution (Theorem \ref{pro:3})
and likelihood ratio test is $4.142$, with p-value $0.042$. Hence,
FTG distribution is a more likelihood model for the data set, since
Pareto distribution is outside of a $95\%$ confidence region for
FTG distribution parameters.

\begin{figure}[h]
\centering{}\centerline{\includegraphics[bb=0bp 302bp 748bp 595bp,clip,width=13cm]{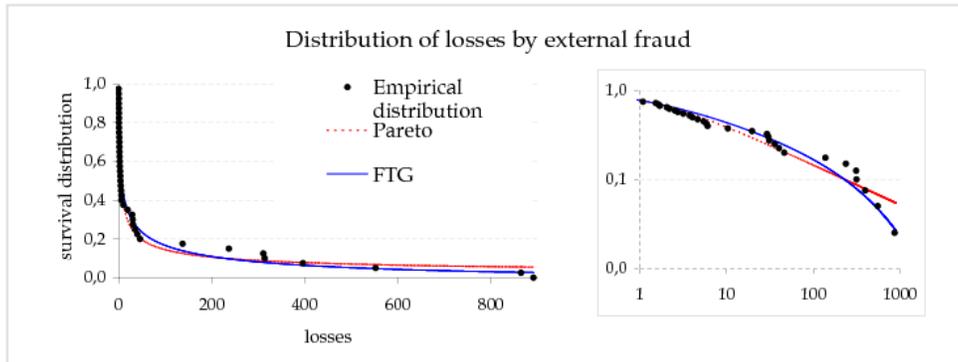}}
\caption{The FTG and the Pareto distributions fit the empirical survival function
in a similar way in the range of the observed sample. However, the
estimated high quantiles differ greatly.\label{ex2_graf}The left
shows survival distributiona and the picture in the right shows the
same plot in common logarithm for the tail of the functions. }
\end{figure}

Figure \ref{ex2_graf} shows the empirical survival (or reliability)
function and its fit given by Pareto and FTG distributions. The probability
to exceed the maximum of the sample is estimated at $5,52\%$ for
the Pareto distribution and $2.65\%$ for the FTG distribution, this
difference does not seem essential. However, the estimation of high
quantiles heavily depends on the model. The $0.999$ quantile is $6.95\times10^{6}$
for the Pareto distribution and $3.93\times10^{3}$ for the FTG distribution.
Moreover, the difference is even greater to calculate the expected
tail loss over this quantile, that is the expected value of a loss
if a tail event does occur; it is $12970.6$ for the FTG distribution,
since (\ref{eq:cond}), and infinite for the Pareto distribution.
Note that these quantities are measured in a monetary unit (as dollars)
to calculate risk capital, hence a factor of $10^{3}$ is really important.

\begin{table}[h]
\begin{centering}
\begin{tabular}{|c|cc|c|ccc|r|}
\hline 
 & \multicolumn{3}{c|}{Pareto distribution} & \multicolumn{4}{c|}{FTG distribution}\tabularnewline
sample & $\alpha$ & $\sigma$ & Risk capital & $\alpha$  & $\log\theta$  & $\log\rho$  & Risk capital\tabularnewline
\hline 
1 & -0,310 & 0,367 & 2,47E+13 & -0,038 & -7,093 & -9,250 & 10832,98\tabularnewline
2 & -0,373 & 1,122 & 3,50E+11 & 0,003 & -6,771 & -8,251 & 9292,72\tabularnewline
3 & -0,410 & 1,719 & 5,36E+10 & -0,106 & -7,341 & -7,792 & 13407,05\tabularnewline
4 & -0,423 & 1,351 & 1,87E+10 & -0,057 & -6,543 & -7,460 & 6934,26\tabularnewline
5 & -0,441 & 2,195 & 1,23E+10 & -0,006 & -6,520 & -7,217 & 7603,78\tabularnewline
6 & -0,460 & 1,205 & 2,63E+09 & -0,298 & -8,039 & -8,287 & 16860,11\tabularnewline
7 & -0,486 & 1,097 & 6,78E+08 & -0,276 & -7,313 & -7,828 & 8921,12\tabularnewline
8 & -0,538 & 1,769 & 1,78E+08 & -0,360 & -7,613 & -7,444 & 10997,30\tabularnewline
9 & -0,612 & 3,923 & 3,86E+07 & -0,257 & -6,723 & -6,141 & 6503,94\tabularnewline
10 & -0,763 & 3,916 & 1,66E+06 & -0,371 & -6,113 & -5,461 & 3276,98\tabularnewline
\hline 
original & -0,448 & 1,382 & 5,78E+09 & -0,197 & -7,325 & -7,754 & 10820,37\tabularnewline
\hline
\multicolumn{1}{c}{} &  & \multicolumn{1}{c}{} & \multicolumn{1}{c}{} &  &  & \multicolumn{1}{c}{} & \multicolumn{1}{r}{}\tabularnewline
\end{tabular}
\par\end{centering}

\centering{}\caption{Parameter estimates and the risk capital from the Pareto distribution
and the FTG distribution for 10 bootstrap samples and the original
data set.\label{ex2_capreg }}
\end{table}

Risk capital\emph{ }has been calculated as $0.999$ quantile of the
aggregate losses, computed from (\ref{eq:ald}), by simulating $10^{5}$
times $N$  loss events, where $N$ is Poisson distributed with parameter
$\lambda=20$ and the loss events, $L_{i}$, are simulated from the
fitted Pareto and FTG distributions. Using the FTG distribution the
risk capital is $10820.4$, using Pareto distribution is $5.78\times10^{9}$.
If our data were in thousands of dollars (probably is greater) the
Pareto estimation of risk capital for a bank is about the same order
as the USA gross domestic product (that is unrealistic), see the last
file in Table \ref{ex2_capreg }.

In order to see the sample dependence of the risk capital estimate,
we generated several \emph{bootstrap} samples of the same size as
the original data set. It is observed immediately, with a small number
of samples, the instability of the risk capital estimates obtained
with the Pareto distribution. However, the estimates obtained with
FTG distribution are much more stable.

Table \ref{ex2_capreg } reports the parameter estimates and the risk
capital from the Pareto distribution and the FTG distribution for
10 bootstrap samples and for the original data set. In all cases risk
capital\emph{ }has been calculated in the same way. Samples were selected
from 100 bootstrap samples, ordered by the parameter $\alpha$, choosing
one out of 10, for more diversity. Note that only sample-2 corresponds
to the truncated gamma distribution (\ref{trg}) and their behaviour
is not different from the rest. The most prominent fact is that, in
addition to the unrealistic risk capital estimation with Pareto distribution,
its estimation is highly unstable, with a factor of $10^{7}$. 

We must remember that just as the extreme levels of energy for the
tropical cyclones are affected by the limits of the Earth, the economy
is also finite. Hence, FTG distribution can be a valuable alternative
to Pareto distribution on operational risk.

\section{Bibliography}
\begin{enumerate}
\item Abramowitz, M. and Stegun, I. A. (1972). \textit{Handbook of Mathematical
Functions with Formulas, Graphs, and Mathematical Tables}. New York:
Dover.
\item Akinsete, A., Famoye, F and Lee, C. (2008) The beta-Pareto distribution.
\textit{Statistics}, 42, 547 - 563.
\item Arnold, B. C. (1983). \textit{Pareto Distributions}. Fairland, Maryland:
Interna- tional Cooperative Publishing House.
\item Balkema, A., and de Haan, L. (1974). Residual life time at great age.
\textit{Annals of Probability}, 2, 792\textendash{}804. 
\item Barndorff-Nielsen, O. and Cox, D. (1994) \textit{Inference and asymptotics.}
Monographs on Statistics and Applied Probability, 52. Chapman \& Hall,
London.
\item Barndorff-Nielsen, O. (1978). \textit{Information and exponential
families in statistical theory.} Wiley Series in Probability and Mathematical
Statistics. Chichester: John Wiley \& Sons.
\item Brown, L. (1986). \textit{Fundamentals of statistical exponential
families with applications in statistical decision theory.} Lecture
Notes Monograph Series, 9. Hayward, CA: Institute of Mathematical
Statistics. 
\item Chapman, D. G. (1956). Estimating the Parameters of a Truncated Gamma
Distribution.\textit{ Ann. Math. Statist.}, 27, 498-506.
\item Choulakian, V., and Stephens, M. A. (2001). Goodness-of-Fit for the
Generalized Pareto Distribution. \emph{Technometrics}, 43, 478 - 484.
\item Clauset, A., Shalizi, C. R. and Newman, M. E. J. (2009). Power- law
Distributions in Empirical Data. \textit{SIAM Review}, 51, 661 - 703.
\item Coles. S. and Sparks (2006). Extreme value methods for modelling historical
series of large volcanic magnitudes. \textit{Statistics in Volcanology},
Spec. Publ. of the Int. Assoc. of Volcanol. and Chem. of the Earths
Inter. Ch. 5.
\item Corral, A., Osso, A. and Llebot, J.E. (2010). Scaling of tropical-cyclone
dissipation. \textit{Nature Physics}, 6, 693 - 696.
\item Davis, H. T. and Michael L. F. (1979). The Generalized Pareto Law
as a Model for Progressively Censored Survival Data. \textit{Biometrika},
66, 299 - 306.
\item Degen, M., Embrechts, P. and Lambrigger, D. (2007). The quantitative
modeling of operational risk: between g-and-h and EVT. \textit{Astin
Bulletin}, 37, 265-291.
\item Den Broeder, G. G. (1955) On parameter estimation for truncated Pearson
type III distributions. \textit{Ann. Math. Statist,} 26, 659 - 663. 
\item Devroye, L. (1986). \emph{Non-Uniform Random Variate Generation}.
Springer-Verlag, New York.
\item Dixit, U. J. and Phal, K. D. (2005). Estimating scale parameter of
a truncated gamma distribution. \textit{Soochwon Journal of Mathematics,}
31, 515-523.
\item Dutta, K. and Perry, J. (2006). \textit{A Tale of Tails: An Empirical
Analysis of Loss Distribution Models for Estimating Operational Risk
Capital}. Federal Reserve Bank of Boston. Working Paper 06-13. 
\item Embrechts, P. Klüppelberg, C. and Mikosch, T. (1997). \textit{Modelling
Extremal Events for Insurance and Finance}. Springer-Verlag, Berlin. 
\item Fisher, R. A. (1922). On the mathematical fundations of theoretical
statistics. \textit{Philos. Trans. Roy. Soc. London}, Ser. A, 222,
309 - 368.
\item Furlan, C. (2010). Extreme value methods for modelling historical
series of large volcanic magnitudes. \textit{Statistical Modelling},
10, 113 - 132.
\item Harter, H. L. (1967). Maximum-likelihood estimation of the parameters
of a four-parameter generalized gamma populaton from complete and
censored samples. \textit{Technometrics}, 9, 159 - 165. 
\item Hegde, L. M. and Dahiya, R. C. (1989). Estimation of the parameters
of a truncated gamma distribution.\textit{ Communications in Statistics:
Theory and Methods}, 18, 561 - 577.
\item Letac, G. (1992).\textit{ Lectures on natural exponential families
and their variance functions.} Monografi{}as de Matemática, 50, IMPA,
Rio de Janeiro.
\item McNeil, A. J., Frey, R. and Embrechts P. (2005). \textit{Quantitative
Risk Management: Concepts, Techniques and Tools}. Princeton University
Press.
\item Moscadelli, M. (2004). \textit{The modelling of operational risk:
experience with the analysis of the data collected by the Basel Committee.}
Economic working papers, 517, Bank of Italy, Economic Research Department.
\item Philippe, A. (1997). Simulation of right and left truncated gamma
distributions by mixtures. \textit{Statistics and Computing}, 7, 173
- 181.
\item Pickands, J. (1975). Statistical inference using extreme order statistics.
\textit{Annals of Statistics}, 3, 119\textendash{}131.
\item R Development Core Team (2010). R: \textit{A Language and Environment
for Statistical Computing.} R Foundation for Statistical Computing,
Vienna, Austria. 
\item Rachev, S. T., Racheva-Iotova, B., Stoyanov, S. (2010). Capturing
fat tails, in Risk. Risk Management, Derivatives and Regulation, May
2010, 72-77.
\item Sornette, D. (2006). \textit{Critical phenomena in natural sciences.
}Springer Berlin Heidelberg New York.
\item Stacy, E. W. (1962). A generalization of the gamma distribution.\textit{
Ann. Math. Stat.}, 33, 1187 - 1192.
\end{enumerate}

\end{document}